\documentclass{amsart}

\usepackage{hyperref}
\usepackage{amssymb}
\usepackage{amsmath}
\usepackage{amsthm}
\usepackage{enumerate}
\usepackage{graphicx}
\usepackage[usenames,dvipsnames]{xcolor}

\theoremstyle{plain}

\newtheorem{theorem}{Theorem}
\newtheorem{proposition}[theorem]{Proposition}

\newtheorem{conjecture}[theorem]{Conjecture}

\theoremstyle{definition}
\newtheorem{definition}[theorem]{Definition}

\theoremstyle{remark}
\newtheorem*{remark}{Remark}

\newtheorem*{acknowledgments}{Acknowledgments}

\newcommand{\RiemannSphere}{\widehat{\mathbb{C}}}

\begin{document}

\title{Conformal models and fingerprints of pseudo-lemniscates}

\date{February 24 2015}
\author[T. Richards]{Trevor Richards}
\address{Department of Mathematics, Washington and Lee University, Lexington, VA 24450-2116, United States.}
\email{richardst@wlu.edu}

\author[M. Younsi]{Malik Younsi}
\thanks{Second author supported by NSERC}
\address{Department of Mathematics, Stony Brook University, Stony Brook, NY 11794-3651, United States.}
\email{malik.younsi@gmail.com}

\begin{abstract}

We prove that every meromorphic function on the closure of an analytic Jordan domain which is sufficiently well-behaved on the boundary is conformally equivalent to a rational map whose degree is smallest possible. We also show that the minimality of the degree fails in general without the boundary assumptions. As an application, we generalize a theorem of Ebenfelt, Khavinson and Shapiro by characterizing fingerprints of polynomial pseudo-lemniscates.

\end{abstract}

\maketitle

The starting point of this paper is the following conjecture on the behavior of holomorphic functions on the closed unit disk.

\begin{conjecture}
\label{ConjRichards}
Let $f$ be holomorphic on a neighborhood of the closure of the unit disk $\mathbb{D}$. Then there is an injective holomorphic function $\phi : \mathbb{D} \to \mathbb{C}$ and a polynomial $p$ such that
$$f=p \circ \phi$$
on $\mathbb{D}$.
\end{conjecture}
In this case, we say that $f$ and $p$ are \textit{conformally equivalent} and that $p$ is a \textit{conformal model} for $f$ on $\mathbb{D}$. Note that $\mathbb{D}$ can be replaced by any Jordan domain, by the Riemann mapping theorem.

The above was conjectured by the first author in~\cite{RI}, motivated by questions arising from the study of level curve configurations of meromorphic functions. The case where $f$ is a finite Blaschke product follows from the work of Ebenfelft, Khavinson and Shapiro~\cite{EKS} on the characterization of the so-called fingerprints of polynomial lemniscates. Proofs of this special case were also given independently by the first author in~\cite{RI} using level curves, and by the second author in~\cite{YOU} using a simple Riemann surfaces welding argument combined with the uniformization theorem.

Conjecture \ref{ConjRichards} was proved in full generality by George~Lowther and David~Speyer on the internet mathematics forum \textit{math.stackexchange.com}, using Lagrange interpolation and the inverse mapping theorem (see \cite{RIC}). Our main motivation for this paper comes from the observation that the proof in~\cite{YOU} of the finite Blaschke product case can easily be generalized to obtain a proof of Conjecture~\ref{ConjRichards}, with some additional assumptions on the behavior of the function $f$ on $\partial \mathbb{D}$. Although this approach yields a slightly weaker result than that of Lowther and Speyer, it has considerable advantages. For instance, the argument also works in the meromorphic setting, in which case the polynomial $p$ must be replaced by a rational map $r$. Furthermore, the proof is more enlightening in the sense that it gives some information on the rational map $r$ thereby obtained.

More precisely, our main result is the following.

\begin{theorem}
\label{MainThm}
Let $D$ be a Jordan domain with analytic boundary and let $f$ be a meromorphic function on some open neighborhood of the closure of $D$. Assume that $f$ has no critical point on $\partial D$ and that $f(\partial D)$ is a Jordan curve whose bounded face contains $0$. Then there is an injective holomorphic function $\phi:D \to \mathbb{C}$ and a rational map $r$ of degree $\max(M,N)$ such that
$$f=r \circ \phi$$
on $D$, where $M$ and $N$ are the number of zeros and the number of poles of $f$ in $D$, counting multiplicity.

The rational map $r$ may be chosen so that all of its zeros and poles are in $\phi(D)$, except for a pole of multiplicity $M-N$ at $\infty$ if $M \geq N$ or a zero of multiplicity $N-M$ at $\infty$ if $N \geq M$.
\end{theorem}

We mention that a similar conformal factorization result was obtained by Jeong and Taniguchi \cite{JEO} in the case where $f$ is a proper holomorphic mapping from a finitely connected domain $D$ onto the unit disk $\mathbb{D}$. Their main goal was to prove a conjecture of Steven Bell saying that every such domain $D$ is conformally equivalent to a domain of the form $r^{-1}(\mathbb{D})$ for some rational map $r$. See also \cite{FBY} for a uniqueness statement and applications to analytic capacity and Ahlfors functions.

Theorem \ref{MainThm} is proved in Section 1. Note that the degree of $r$ is best possible. It thus seems natural to expect that in Conjecture~\ref{ConjRichards}, one can also obtain a conformal model for $f$ whose degree depends only on the maximal number of preimages under $f$. However, we show in Section~2 that this is not the case.

Finally, in the last section of the paper, we use Theorem \ref{MainThm} to characterize the welding diffeomorphisms (or fingerprints) of \textit{polynomial pseudo-lemniscates}, that is, analytic Jordan curves of the form $p^{-1}(\Gamma)$ where $p$ is a polynomial and $\Gamma$ is an analytic Jordan curve. This generalizes the main theorem of~\cite{EKS}.

\section{Proof of Theorem \ref{MainThm} }
This section contains the proof of Theorem \ref{MainThm}. First, we need some notation.

Let $D$ and $f$ be as in Theorem \ref{MainThm}. We define $U$ and $V$ to be the bounded and unbounded faces of $E:=f(\partial D)$ respectively. We also let $g:\mathbb{D}\to U$ and $h:\mathbb{D}_+ \to V$ be conformal maps, normalized so that $g(0)=0$ and $h(\infty)=\infty$. Here $\mathbb{D}_+$ denotes the complement of the closed unit disk $\overline{\mathbb{D}}$ in the Riemann sphere $\RiemannSphere$. Note that since $E$ is analytic, the maps $g$ and $h$ extend to conformal maps on a neighborhood of the closure of their respective domains.

We now proceed with the proof of Theorem \ref{MainThm}.

\begin{proof}
We can assume that $D$ is the unit disk $\mathbb{D}$, replacing $f$ by $f \circ \psi$ where $\psi:\mathbb{D} \to D$ is conformal if necessary.  Moreover, we can suppose that $M \geq N$, since otherwise it suffices to replace $f$ by $1/f$. Let $n$ be the degree of the covering map $f:\mathbb{T} \to E$, where $\mathbb{T}=\partial \mathbb{D}$ is the unit circle. By the argument principle, we have $n=1/2\pi i \int_\mathbb{T} f'(\zeta)/f(\zeta) \, d\zeta=M-N$, so that $n \geq 0$. Furthermore, the fact that $f' \neq 0$ on $\mathbb{T}$ implies that $n \geq 1$. Define $A:=f \circ {g}^{-1} : \overline{U} \to \RiemannSphere$ and $B:=h \circ S \circ {h}^{-1} : \overline{V} \to \overline{V}$, where $S(z):=z^n$. Then the restrictions of $A$ and $B$ to $E$ are both covering maps of degree $n$ of $E$ onto itself. It follows from the basic theory of covering spaces that there exists a homeomorphism $C:E \to E$ such that $A \circ C = B$. Clearly, $C$ extends analytically to a neighborhood of $E$. Consider the Riemann surface $X:=\overline{U} \sqcup\overline{V} / \sim_C$ formed by welding conformally a copy of $\overline{V}$ to $\overline{U}$ using the analytic homeomorphism $C$ on $E$. Topologically, $X$ is the connected sum of $\overline{U}$ with the closed topological disk $\overline{V}$, so it is homeomorphic to a sphere. By the uniformization theorem, there exists a biholomorphism $G:X \to \RiemannSphere$ with $G(\infty)=\infty$.  Now, define $F : X \to \RiemannSphere$ by
$$
F(z) := \left\{ \begin{array}{rl} A(z)  &  \mbox{for } z\in \overline{U} \\ B(z) & \mbox{for }z \in \overline{V} \end{array} \right..
$$

Note that the map $F$ is well-defined because $A \circ C = B$ on $E$. Furthermore, it follows from Morera's theorem that $F$ is holomorphic on $X$, since $\partial U = E$ is an analytic curve. The composition $F \circ G^{-1} : \RiemannSphere \to \RiemannSphere$ is meromorphic and thus equal to a rational map $r$. On $U$, we have $r \circ G = F = A = f \circ {g}^{-1}$. Hence the existence part of the theorem follows by setting $\phi:=G \circ g$.

Finally, it follows directly from the construction that all of the zeros and poles of $r$ are in $\phi(\mathbb{D})$, except for a pole of multiplicity $M-N$ at $\infty$. Moreover, since $r=f \circ \phi^{-1}$ on $\phi(\mathbb{D})$, the poles of $r$ in $\mathbb{C}$ are the $N$ images of the poles of $f$ under $\phi$ (counting multiplicity), from which we deduce that $r$ has degree $M-N+N=M$.

This completes the proof of the theorem.

\end{proof}

\begin{remark}
If $f$ is holomorphic, then $N=0$ and all of the poles of $r$ are at $\infty$, so that $r=p$, a polynomial of degree $M$.
\end{remark}

\section{Degree of the conformal model}

In the previous section, we proved that every meromorphic function $f$ on the closure of an analytic Jordan domain $D$ which is  sufficiently well-behaved on the boundary of $D$ is conformally equivalent to a rational map $r$. Moreover, we also showed that $r$ can be taken so that its degree is smallest possible, namely the maximum between the number of zeros and the number of poles of $f$ in $D$. One can easily check that under the assumptions on $f$ in Theorem~\ref{MainThm}, this maximum also equals the maximal number of preimages (counting multiplicity) under $f$ in $D$ of any point in $\RiemannSphere$.

Assume now that $f$ is holomorphic on the closure of a Jordan domain. Even without any assumption on the boundary behavior of $f$, we know by the proof of Conjecture~\ref{ConjRichards} in \cite{RIC} that there exists a polynomial conformal model for $f$. Our goal in this section is to show that in this case, the polynomial cannot be taken in general to have smallest possible degree.

First, we make rigorous this notion of smallest possible degree in terms of the function $f$.

\begin{definition}
For any set $E \subset \RiemannSphere$ and any function $f:E \to \RiemannSphere$, we define the \textit{degree of non-injectivity} of $f$ in $E$ by
$$\operatorname{N}(f,E):= \sup_{w \in \RiemannSphere} n(w),$$
where $n(w)$ is the number of preimages of $w$ under $f$ in $E$, counted with multiplicity.
\end{definition}

Note that if $D \subset \mathbb{C}$ is a Jordan domain and $f$ is meromorphic and non-constant on some open neighborhood of $\overline{D}$, then $\operatorname{N}(f,D)$ is finite, by the identity theorem. Moreover, it is easy to see that if $r$ is a rational conformal model for $f$ on $D$, then $\deg(r)\geq \operatorname{N}(f,D)$. It has been a long-standing conjecture of the first author that any such function $f$ can be conformally modeled by a rational function (or a polynomial in the analytic case) of degree precisely equal to $\operatorname{N}(f,D)$.

However, in this section we give a counterexample to this conjecture by constructing, for each $n \geq 4$, a function $f_n$ which is analytic on the closure of a Jordan domain $D_n$ and satisfies $\operatorname{N}(f_n,D_n)=2$, but which cannot be conformally modeled on $D_n$ by a polynomial of degree less than $n$.  For simplicity, we will only describe the construction in the case $n=4$, and indicate the method of extension to higher values of $n$.

Note that we will use throughout this construction the general fact that if $f$ is meromorphic on a neighborhood of the closure of a Jordan domain $D$ and $r$ is a rational conformal model for $f$ on $D$ with conformal map $\phi:D\to\phi(D)$ satisfying $f=r\circ\phi$ on $D$, then the map $\phi$ extends continuously to $\partial D$ and the functional equation $f=r\circ\phi$ holds on $\partial D$ as well.  This fact may be readily checked by a routine analytic continuation argument.

\subsection{The Construction}

Define $f_4(z)=z^2(z+1)(z+3)$. Then the critical points of $f_4$ are $0$ and $(-3\pm\sqrt{3})/2$. In Figure~\ref{fig:D_4}, we see the critical level curves of $f_4$ (i.e. the level curves containing the critical points), along with a shaded region, denoted by $D_4$. The smooth pieces of $\partial D_4$ alternate between portions of level curves of $f_4$ and portions of gradient lines of $f_4$. Our goal is to establish the following facts :
\begin{enumerate}
    \item $\operatorname{N}(f_4,D_4)=2$.
    \item For any polynomial conformal model $p$ of $f_4$ on $D_4$, we have $\deg(p)\geq4$.
\end{enumerate}

\begin{figure}[htb!]
\centering
\includegraphics[width=0.75\textwidth]{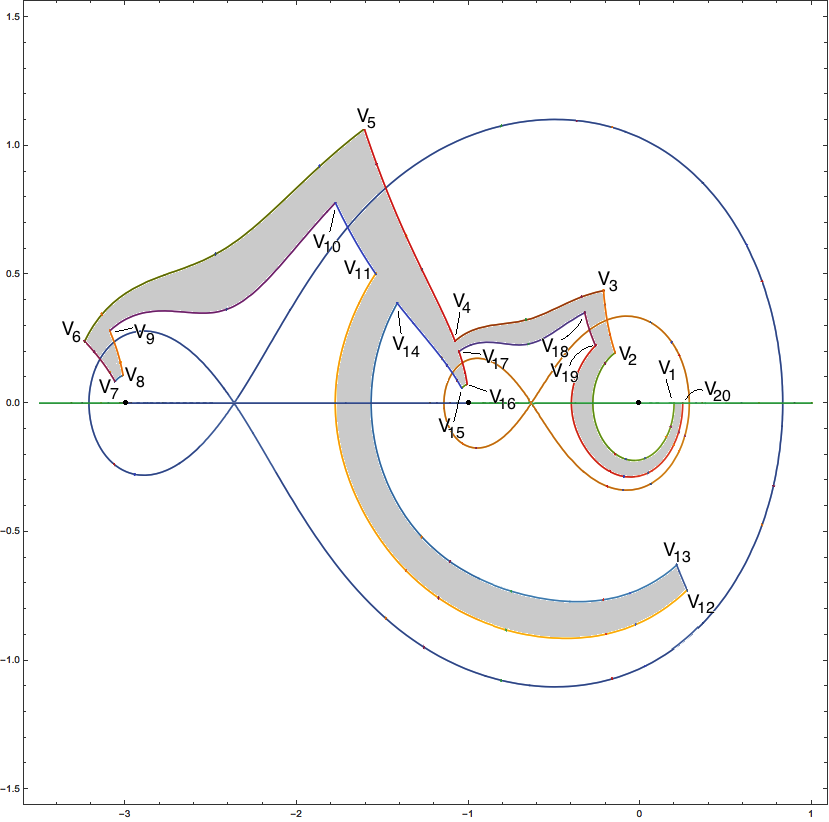}
\caption{The domain $D_4$.}
\label{fig:D_4}
\end{figure}

In order to establish these facts, it is important to determine the precise values of $|f_4|$ and $\arg(f_4)$ on each of the different pieces of $\partial D_4$. We begin by considering a single piece of $\partial D_4$, namely the one with endpoints $v_5$ and $v_6$, which we denote by $\overline{v_5v_6}$. First, we set up some notation.

\begin{definition}
Let $f$ be meromorphic on a domain $D$. Then for $\epsilon>0$ and $\alpha\in [0,2\pi)$, we define
$$\operatorname{Lev}(f,\epsilon)=\{z\in D:|f(z)|=\epsilon\}$$
and
$$\operatorname{Grad}(f,\alpha)=\{z\in D:\arg(f(z))=\alpha\}.$$
\end{definition}

Now, the piece $\overline{v_5v_6}$ is a portion of $\operatorname{Lev}(f_4,8)$. The endpoint $v_5$ is an intersection point between $\operatorname{Lev}(f_4,8)$ and $\operatorname{Grad}(f_4,\pi/2)$, while the other endpoint $v_6$ is an intersection point between $\operatorname{Lev}(f_4,8)$ and $\operatorname{Grad}(f_4,5\pi/3)$. The sets $\operatorname{Lev}(f_4,8)$, $\operatorname{Grad}(f_4,\pi/2)$, and $\operatorname{Grad}(f_4,5\pi/3)$ are depicted in Figure~\ref{fig:v_5v_6 isolated}.  In that figure, $\operatorname{Lev}(f_4,8)$ is the smooth closed curve, and the rays of $\operatorname{Grad}(f_4,\pi/2)$ and $\operatorname{Grad}(f_4,5\pi/3)$ are marked with their corresponding argument.

\begin{figure}[ht]
\centering
\includegraphics[width=0.75\textwidth]{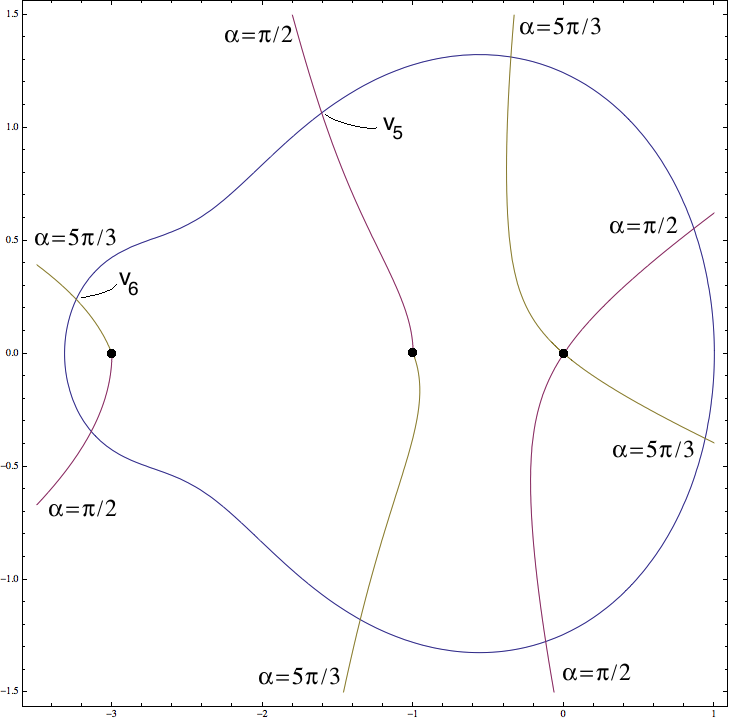}
\caption{The portion $\overline{v_5v_6}$.}
\label{fig:v_5v_6 isolated}
\end{figure}

Note that traversing the portion $\overline{v_5v_6}$ from $v_5$ to $v_6$ corresponds to traversing the level curve $\operatorname{Lev}(f_4,8)$ with positive orientation. In other words, if $z$ traverses $\overline{v_5v_6}$ from $v_5$ to $v_6$, then $\arg(f_4(z))$ is increasing. Since $\arg(f_4(v_5))=\pi/2$ and $\arg(f_4(v_6))=5\pi/3$, the total change in $\arg(f_4)$ along $\overline{v_5v_6}$ is $7\pi/6+2\pi k$ for some non-negative integer $k$. But since $\arg(f_4)$ does not take the value $\pi/2$ at any point of $\overline{v_5v_6}$ other than $v_5$, the total change in $\arg(f_4)$ along $\overline{v_5v_6}$ from $v_5$ to $v_6$ is exactly $7\pi/6$.

We apply a similar analysis to each level curve portion of $\partial D_4$ and record the resulting data in Table~1.  Note that for $\overline{v_1v_2}$ as well as five other of the level curve portions of $\partial D_4$, the total change in $\arg(f_4)$ is negative. This reflects the fact that, when $\overline{v_1v_2}$ is traversed from $v_1$ to $v_2$, the level set of $f_4$ which contains $\overline{v_1v_2}$ is being traversed with negative orientation.

\begin{table}[!ht]
\centering
\label{table: Function data.}\caption{Values of $f_4$ on $\partial D_4$.}

\begin{tabular}{ |l|l|l|l|l| }

  \hline
Segment&$|f_4|$&Total change in $\arg(f_4)$&Initial $\arg(f_4)$&Final $\arg(f_4)$\\
\hline
$\overline{v_1v_2}$&$0.15$&$-5\pi/2$&$0$&$3\pi/2$\\
\hline
$\overline{v_3v_4}$&$0.6$&$\pi$&$3\pi/2$&$\pi/2$\\
\hline
$\overline{v_5v_6}$&$8$&$7\pi/6$&$\pi/2$&$5\pi/3$\\
\hline
$\overline{v_7v_8}$&$2$&$-\pi/6$&$5\pi/3$&$3\pi/2$\\
\hline
$\overline{v_9v_{10}}$&$6$&$-5\pi/6$&$3\pi/2$&$2\pi/3$\\
\hline
$\overline{v_{11}v_{12}}$&$3$&$7\pi/3$&$2\pi/3$&$\pi$\\
\hline
$\overline{v_{13}v_{14}}$&$2$&$-7\pi/3$&$\pi$&$2\pi/3$\\
\hline
$\overline{v_{15}v_{16}}$&$0.15$&$-\pi/6$&$2\pi/3$&$\pi/2$\\
\hline
$\overline{v_{17}v_{18}}$&$0.47$&$-5\pi/6$&$\pi/2$&$5\pi/3$\\
\hline
$\overline{v_{19}v_{20}}$&$0.25$&$7\pi/3$&$5\pi/3$&$0$\\
\hline
\end{tabular}

\end{table}

We now use the data in Table 1 to show that $\operatorname{N}(f_4,D_4)=2$ and that $f_4$ cannot be conformally modeled on $D_4$ by a polynomial of degree less than $4$.

First, to prove that $\operatorname{N}(f_4,D_4)=2$, fix some value $k$ in $(0.15,8)$, which is the range of values of $|f_4|$ on $D_4$. We have to show that for every $\theta\in[0,2\pi)$, the function $f_4$ takes the value $ke^{i\theta}$ at most twice.

Assume that $k\in(0.15,0.25)$. There are two portions of $\operatorname{Lev}(f_4,k)$ in $D_4$, namely the one connecting $\overline{v_{2}v_{3}} \subset \operatorname{Grad}(f_4,3\pi/2)$ to $\overline{v_{1}v_{20}} \subset \operatorname{Grad}(f_4,0)$, which we denote by $L_1$, and the one connecting $\overline{v_{14}v_{15}} \subset \operatorname{Grad}(f_4,2\pi/3)$ to $\overline{v_{16}v_{17}} \subset \operatorname{Grad}(f_4,\pi/2)$, which we denote by $L_2$.

Now, note that the fact that $f_4$ has no critical point in $D_4$ implies that $\arg(f_4(z))$ is strictly monotonic as $z$ traverses either $L_1$ or $L_2$. Since the total change in $\arg(f_4)$ along $L_1$ is $-5\pi/2$ (i.e. the same as the total change in $\arg(f_4)$ along $\overline{v_1v_2}$), it follows that $\arg(f_4)$ takes each value in $[0,2\pi)$ at most twice on $L_1$. On the other hand, $\arg(f_4)$ does take some values twice on $L_1$, namely those in the interval $(3\pi/2,2\pi)$, from which we deduce that
$N(f_4,D_4)\geq2$. To prove the reverse inequality, first observe that on $L_2$, $\arg(f_4)$ takes each value in the interval $(\pi/2,2\pi/3)$ exactly once. Since this interval is disjoint from $(3\pi/2,2\pi)$, we get that $\arg(f_4)$ takes each value in $[0,2\pi)$ at most twice on the union $L_1\cup L_2$. In other words, the function $f_4$ takes the value $ke^{i\theta}$ at most twice in $D_4$, for each $\theta\in[0,2\pi)$. As every other value of $k\in[0.15,8]$ may be treated similarly, we conclude that $\operatorname{N}(f_4,D_4)=2$.

Suppose now that $p$ is a polynomial conformal model for $f_4$ on $D_4$.  Let $\phi:D_4\to\mathbb{C}$ be an injective holomorphic function such that $f_4=p\circ\phi$ on $D_4$. Let $\Lambda_1$ and $\Lambda_1'$ denote the (not necessarily distinct) components of $\operatorname{Lev}(p,0.15)$ containing $\phi\left(\overline{v_1v_2}\right)$ and $\phi\left(\overline{v_{15}v_{16}}\right)$ respectively. Also, let $\Lambda_2$ and $\Lambda_2'$ denote the components of $\operatorname{Lev}(p,2)$ containing $\phi\left(\overline{v_{13}v_{14}}\right)$ and $\phi\left(\overline{v_{7}v_{8}}\right)$ respectively.

Our goal is to show that $\Lambda_1\neq\Lambda_1'$ and $\Lambda_2\neq\Lambda_2'$.  Suppose for the moment that this is true.

Let $F_1$ denote the bounded face of $\Lambda_1$ incident to $\phi(\overline{v_1v_2})$. Since the absolute value of the change in $\arg(f_4)$ along $\overline{v_1v_2}$ is $5\pi/2$, it follows that the change in $\arg(p)$ along $\partial F_1$ is greater than or equal to $5\pi/2$, and thus $p$ has at least two zeros in $F_1$. Since $\Lambda_1\neq\Lambda_1'$, the face $F_1'$ of $\Lambda_1'$ which is incident to $\phi(\overline{v_{15}v_{16}})$ must be disjoint from $F_1$, and the maximum modulus theorem then implies that $F_1'$ contains at least one other zero of $p$ than those in $F_1$.

Now, let $F_2$ and $F_2'$ denote the bounded faces of $\Lambda_2$ and $\Lambda_2'$ which are incident to $\phi\left(\overline{v_{13}v_{14}}\right)$ and $\phi\left(\overline{v_7v_8}\right)$ respectively. Inspection of $D_4$ shows that $\overline{v_1v_2}$ and $\overline{v_{15}v_{16}}$ are both contained in the same bounded face of the level curve of $f_4$ which contains $\overline{v_{13}v_{14}}$ (namely the face which is incident to $\overline{v_{13}v_{14}}$), and therefore $\Lambda_1$ and $\Lambda_1'$ are both contained in the bounded face of $\Lambda_2$ which is incident to $\phi\left(\overline{v_{13}v_{14}}\right)$ (namely $F_2$). The maximum modulus theorem implies that $F_2'$ contains at least one zero and, since $\Lambda_2\neq\Lambda_2'$, $F_2$ and $F_2'$ are disjoint. In other words, the zero of $p$ in $F_2'$ is not one of the three zeros of $p$ already known to reside in $F_2$. We conclude that $\deg(p)\geq4$.

It thus remains to show that $\Lambda_1\neq\Lambda_1'$ (the fact that $\Lambda_2\neq\Lambda_2'$ can be proved by a similar argument).

Suppose, in order to obtain a contradiction, that $\Lambda_1=\Lambda_1'$. Let $\gamma_1$ be the negatively-oriented path in $\Lambda_1$ starting at $\phi(v_{16})$ and intersecting $\phi(\overline{v_1v_2})$ only at one point. Then since $|p \circ \phi| = |f_4| > 0.15$ on $\overline{D_4} \setminus (\overline{v_{1}v_{2}} \cup \overline{v_{15}v_{16}})$, we get that $\gamma_1$ is contained in the complement of $\phi\left(\overline{D_4}\right)$, except at the endpoints. Also, by analyticity of $p$, we have that $\arg{p(z)}$ is decreasing as $z$ traverses $\gamma_1$. The fact that the net change in $\arg(f_4)$ along $\overline{v_1v_2}$ is negative, as shown in Table 1, then implies that the other endpoint of $\gamma_1$ must be $\phi(v_1)$.

Consider now the closed path $\gamma_2$ starting at $\phi(v_{16})$ obtained by traversing in order $\gamma_1$, $\phi\left(\overline{v_1v_{20}}\right)$, $\phi\left(\overline{v_{20}v_{19}}\right)$, $\phi\left(\overline{v_{19}v_{18}}\right)$, $\phi\left(\overline{v_{18}v_{17}}\right)$ and $\phi\left(\overline{v_{17}v_{16}}\right)$. Note that since $|p|\leq0.47$ on $\gamma_2$ and $|p|>0.47$ at some points in $\phi(D_4)$, we get that $\phi(D_4)$ is contained in the unbounded face of $\gamma_2$.  Define $F_3$ to be the bounded face of $\gamma_2$ which is incident to $\phi\left(\overline{v_{17}v_{18}}\right)$.  We will show that the net change in $\arg(p)$ along $\gamma_2$ is negative. Since $\gamma_2$ traverses $\partial F_3$ with positive orientation and $p$ has no pole in $F_3$, this will contradict the argument principle.

Let $\delta\in\mathbb{R}$ be the net change in $\arg(p(z))$ as $z$ traverses the path $\gamma_1$.  As noted above, $\arg(p)$ is strictly decreasing along $\gamma_1$, so $\delta<0$. By examining the other portions of the path $\gamma_2$, we get that the net change in $\arg(p(z))$ as $z$ traverses $\gamma_2$ starting at $\phi(v_{16})$ is

$$\delta+0+\dfrac{-7\pi}{3}+0+\dfrac{5\pi}{6}+0=\delta-\dfrac{3\pi}{2}<0.$$
As mentioned above, this contradicts the analyticity of $p$.  It follows that $\Lambda_1\neq\Lambda_1'$.  This completes the proof that $f_4$ cannot be modeled on $D_4$ by a polynomial of degree less than $4$.

For the case $n>4$, it suffices to replace the polynomial $f_4(z)=z^2(z+1)(z+3)$ by a degree $n$ polynomial $f_n$ with a similar critical level curve configuration. More precisely, the polynomial $f_n$ should have a zero of multiplicity two at the origin and $n-2$ more simple zeros located on the negative real axis so that the critical level curve configuration consists of a nested sequence of figure-eight graphs. It was proved in~\cite{RI} that every possible critical level curve configuration is instantiated by some polynomial, so we may let $f_n$ be the polynomial with this configuration, and choose $D_n$ similarly as in the case $n=4$ above.

\section{Fingerprints of polynomial pseudo-lemniscates}

In this section, we describe how Theorem \ref{MainThm} can be applied to the conformal welding problem for polynomial pseudo-lemniscates.

Let $\Gamma$ be a smooth Jordan curve in the plane. Denote by $\Omega_{-}$ and $\Omega_{+}$ respectively the bounded and unbounded faces of $\Gamma$. Also, let $\phi_{-}:\mathbb{D} \to \Omega_{-}$ and $\phi_+ : \mathbb{D}_{+} \to \Omega_{+}$ be conformal maps (recall that $\mathbb{D}_{+}=\RiemannSphere \setminus \overline{\mathbb{D}}$). We assume that $\phi_{+}$ is normalized by $\phi_{+}(\infty)=\infty$ and $\phi_{+}'(\infty)>0$, the latter meaning that $\phi_{+}$ has a Laurent development of the form
$$\phi_{+}(z)=az+a_0+\frac{a_1}{z}+\dots$$
near $\infty$, for some $a>0$. With this normalization, the map $\phi_+$ is unique.

It is well-known that $\phi_{+}$ and $\phi_{-}$ extend to diffeomorphisms on the closure of their respective domains, so that we can consider the map $k_{\Gamma}:=\phi_{+}^{-1} \circ \phi_{-} : \mathbb{T} \to \mathbb{T}$, an orientation-preserving diffeomorphism of the unit circle $\mathbb{T}$ onto itself. The map $k_{\Gamma}$ is called the \textit{fingerprint} or \textit{conformal welding diffeomorphism} of $\Gamma$ and is uniquely determined up to precomposition with an automorphism of the unit disk, that is, a degree one Blaschke product. Furthermore, the fingerprint $k_\Gamma$ is invariant under translations and scalings of the curve $\Gamma$. We thus obtain a map $\mathcal{F}$ from the set of equivalence classes of smooth curves, modulo translations and scalings, into the set of equivalence classes of orientation-preserving diffeomorphisms of the unit circle onto itself, modulo automorphisms of $\mathbb{D}$.

The following result is a consequence of the so-called \textit{conformal welding theorem} of Pfluger \cite{PFL} based on the theory of quasiconformal mappings. It is also sometimes referred to as Kirillov's theorem.

\begin{theorem}
\label{WeldThm}
The map $\mathcal{F}$ is a bijection.
\end{theorem}
The injectivity of $\mathcal{F}$ follows easily from the fact that smooth Jordan curves are conformally removable, see e.g. \cite{YOU}. For more information on conformal welding, we refer the reader to \cite[Chapter 2, Section 7]{LEV} and \cite{HAM}.

In recent years, motivated mainly by applications to the field of computer vision and pattern recognition, the problem of recovering the curve $\Gamma$ from its fingerprint $k_\Gamma$ has been extensively studied. In practice, there are several numerical methods that can be used, see for instance \cite{SHM} and \cite{MAR}. In theory, however, it is very difficult to explicitly determine the curve corresponding to a given fingerprint. A more convenient approach is to study the fingerprints of curves belonging to various special families. This was instigated by Ebenfelt, Khavinson and Shapiro \cite{EKS}, who characterized the fingerprints of polynomial lemniscates. More precisely, let $p$ be a (complex) polynomial of degree $n$. The \textit{lemniscate} of $p$ is defined by
$$L(p):=\{z \in \RiemannSphere : |p(z)|=1\} = p^{-1}(\mathbb{T}).$$
We further define
$$\Omega_-(p):= \{ z \in \RiemannSphere : |p(z)|<1 \} = p^{-1}(\mathbb{D})$$
and
$$\Omega_+(p):= \{z \in \RiemannSphere : |p(z)|>1 \} = p^{-1}(\mathbb{D}_+).$$
It follows from the maximum modulus principle that $\Omega_+(p)$ is connected. We say that the lemniscate $L(p)$ is \textit{proper} if $\Omega_{-}(p)$ is connected (and therefore simply connected).

Assume now that $L(p)$ is proper and, without loss of generality, that the leading coefficient of $p$ is positive. Then $L(p)$ is a smooth Jordan curve and therefore has a fingerprint $k:\mathbb{T} \to \mathbb{T}$. The following result characterizes exactly which diffeomorphisms arise in this way.

\begin{theorem}
\label{thmEKS}
The fingerprint $k:\mathbb{T} \to \mathbb{T}$ of $L(p)$ is given by
$$k(z):=B(z)^{1/n},$$
where $B$ is a Blaschke product of degree $n$ whose zeros are the preimages of the zeros of $p$ under a conformal map of $\mathbb{D}$ onto $\Omega_-(p)$, counted with multiplicity. Conversely, given any Blaschke product $B$ of degree $n$, there is a polynomial $p$ of the same degree with positive leading coefficient such that the lemniscate $L(p)$ is proper and has $k=B^{1/n}$ as its fingerprint. Moreover, the polynomial $p$ is unique up to precomposition with a linear map of the form $T(z)=az+b$, where $a>0$ and $b \in \mathbb{C}$.
\end{theorem}

This theorem was first proved in \cite{EKS}. A simpler proof using the uniformization theorem was later given in~\cite{YOU}, together with a generalization to rational lemniscates.

Our goal now is to use Theorem \ref{MainThm} to generalize Theorem \ref{thmEKS} to polynomial pseudo-lemniscates. First, we need some definitions. Let $\Gamma$ be an analytic Jordan curve in the plane and let $\Omega_{-}$, $\Omega_{+}$ and $\phi_{-}:\mathbb{D} \to \Omega_-$, $\phi_+: \mathbb{D}_+ \to \Omega_+$ be as before. Recall that $\phi_+$ is normalized so that $\phi_+(\infty)=\infty$ and $\phi_+'(\infty)>0$. We also assume that $0 \in \Omega_-$ and that $\phi_-$ is normalized by $\phi_-(0)=0$ and $\phi_-'(0)>0$. Denote by $k_{\Gamma}:=\phi_{+}^{-1} \circ \phi_{-} : \mathbb{T} \to \mathbb{T}$ the fingerprint of $\Gamma$.

Now, let $p$ be a complex polynomial of degree $n$. Assume that the leading coefficient of $p$ is positive. We define the \textit{pseudo-lemniscate} of $p$ corresponding to the curve $\Gamma$ by
$$L(p,\Gamma):=\{z \in \RiemannSphere : p(z) \in \Gamma\} = p^{-1}(\Gamma).$$
Note that with this definition, the original lemniscate $L(p)$ is just $L(p,\mathbb{T})$. We also define
$$\Omega_-(p,\Gamma):= p^{-1}(\Omega_-)$$
and
$$\Omega_+(p,\Gamma):= p^{-1}(\Omega_+),$$
so that $\Omega_-(p)=\Omega_-(p,\mathbb{T})$ and $\Omega_+(p)=\Omega_+(p,\mathbb{T}).$

As before, the maximum modulus principle implies that $\Omega_+(p,\Gamma)$ is connected. We say that the pseudo-lemniscate $L(p,\Gamma)$ is \textit{proper} if $\Omega_-(p,\Gamma)$ is connected.

The following proposition characterizes exactly when this happens. It was proved in \cite[Propositon 2.1]{EKS} in the case $\Gamma=\mathbb{T}$.

\begin{proposition}
The following are equivalent:
\begin{enumerate}[\rm(i)]
\item The pseudo-lemniscate $L(p,\Gamma)$ is proper.
\item All of the $n-1$ critical values of $p$ in $\mathbb{C}$ (counted with multiplicity) belong to $\Omega_{-}$.
\end{enumerate}
In this case, the pseudo-lemniscate $L(p,\Gamma)$ is an analytic Jordan curve.
\end{proposition}

\begin{proof}
Let $X:=\overline{\Omega_+(p,\Gamma)}$. Then $X$ is connected. Let $k$ denote the number of components of $\Omega_-(p,\Gamma)$. Then $X$ is topologically a sphere with $k$ disks removed, so its Euler characteristic is equal to $2-k$. The Riemann-Hurwitz formula therefore gives
\begin{eqnarray*}
2-k&=&\operatorname{deg}(p)\chi(\overline{\Omega_+})-|\{\mbox{critical points of}\,\, p \,\, \mbox{in} \,\, X\}|\\
&=& n-|\{\mbox{critical points of}\,\, p \,\, \mbox{in} \,\, X\}|,\\
\end{eqnarray*}
so that $k=1$ if and only if $p$ has $n-1$ critical points in $X$. But $p$ has $2n-2$ critical points in $\RiemannSphere$ (including the critical point at $\infty \in X$ of multiplicity $n-1$) and therefore $\Omega_-(p,\Gamma)$ is connected if and only if $p$ has $n-1$ critical points in $\Omega_-(p,\Gamma)$, which is what we had to prove. Finally, if this is the case, then $L(p,\Gamma)$ is an analytic Jordan curve since the restriction $p:L(p,\Gamma) \to \Gamma$ is non-singular.

\end{proof}

Assume now that the pseudo-lemniscate $L(p,\Gamma)$ is proper. Let $\phi_{p-} : \mathbb{D} \to \Omega_-(p,\Gamma)$ and $\phi_{p+}:\mathbb{D}_+ \to \Omega_+(p,\Gamma)$ be conformal maps with $\phi_{p+}(\infty)=\infty$ and $\phi_{p+}'(\infty)>0$. Also, let $k_p:=\phi_{p+}^{-1} \circ \phi_{p-}$ be the fingerprint of $L(p,\Gamma)$. The following theorem characterizes $k_p$ in terms of the fingerprint $k_\Gamma$ of $\Gamma$.

\begin{theorem}
The fingerprint $k_p:\mathbb{T} \to \mathbb{T}$ of the pseudo-lemniscate $L(p,\Gamma)$ is given by
$$k_p=(k_\Gamma \circ B)^{1/n},$$
where $B$ is a Blaschke product of degree $n$ whose zeros are the preimages under $\phi_{p-}$ of the zeros of $p$, counted with multiplicity. Conversely, given any Blaschke product $B$ of degree $n$, there is a polynomial $p$ of the same degree with positive leading coefficient such that the pseudo-lemniscate $L(p,\Gamma)$ is proper and has $k_p=(k_\Gamma \circ B)^{1/n}$ as its fingerprint. Moreover, the polynomial $p$ is unique up to precomposition with a linear map of the form $T(z)=az+b$, where $a>0$ and $b \in \mathbb{C}$.
\end{theorem}

Note that if $\Gamma=\mathbb{T}$, then $k_\Gamma$ is the identity on $\mathbb{T}$ and the theorem reduces to Theorem \ref{thmEKS}.

\begin{proof}
To prove the first statement, note that $\phi_-^{-1} \circ p \circ \phi_{p-}$ is a degree $n$ proper holomorphic map of $\mathbb{D}$ onto itself, and hence must be a Blaschke product $B$ of degree $n$. Since $\phi_-(0)=0$, the zeros of $B$ are precisely the preimages under $\phi_{p-}$ of the zeros of $p$. Similarly, the map $\phi_+^{-1} \circ p \circ \phi_{p+}$ is also a Blaschke product, having all of its poles at $\infty$ in view of the normalizations of $\phi_+$ and $\phi_{p+}$. Hence $\phi_+^{-1} \circ p \circ \phi_{p+}(z)=cz^n$ for some unimodular constant $c$. Since $\phi_{p+}'(\infty), \phi_+'(\infty)>0$ and the leading coefficient of $p$ is positive, we get that $c=1$, thus $\phi_+^{-1} \circ p \circ \phi_{p+}=S$ where $S(z)=z^n$. Composing both sides of the equation $\phi_{p+} \circ k_p = \phi_{p-}$ by $\phi_{+}^{-1} \circ p$, we get
$$S \circ k_p = \phi_+^{-1} \circ p \circ \phi_{p-} = \phi_+^{-1} \circ \phi_- \circ B = k_\Gamma \circ B,$$
which proves the first statement.

To prove the second statement, apply the remark following the proof of Theorem \ref{MainThm} to $f=\phi_- \circ B$ to obtain a polynomial $p$ of degree $n$ and a conformal map $\phi_{p-}$ on $\mathbb{D}$ such that
$$p \circ \phi_{p-}=\phi_- \circ B.$$

We can assume that the leading coefficient of $p$ is positive. Since $p$ and $B$ have the same degree, we have
$$\phi_{p-}(\mathbb{D}) = p^{-1}(\Omega_-) = \Omega_-(p,\Gamma)$$
and thus the pseudo-lemniscate $L(p,\Gamma)$ is proper. Now, let $\phi_{p+}:\mathbb{D}_+ \to \Omega_+(p,\Gamma)$ be the Riemann map normalized by $\phi_{p+}(\infty)=\infty$ and $\phi_{p+}'(\infty)>0$. As in the first part of the proof, we get that $\phi_+^{-1} \circ p \circ \phi_{p+}=S$ and
$$S \circ k_p = \phi_+^{-1} \circ p \circ \phi_{p-} = \phi_+^{-1} \circ \phi_- \circ B = k_\Gamma \circ B.$$
This proves the second part of the theorem.

Finally, to prove the uniqueness statement, assume that $q$ is a polynomial of degree $n$ with positive leading coefficient such that $L(q,\Gamma)$ is proper and $k_q=(k_\Gamma \circ B)^{1/n}$. By Theorem \ref{WeldThm}, we have $L(p,\Gamma)=T(L(q,\Gamma))$ for some linear map $T(z)=az+b$, where $a>0$ and $b \in \mathbb{C}$. This implies that $L(q,\Gamma)=L(p\circ T,\Gamma)$, i.e. the polynomials $q$ and $p \circ T$ have the same pseudo-lemniscate. Moreover, both these polynomials have degree $n$ so that the map
$$\psi:=\frac{\phi_+^{-1} \circ q}{\phi_+^{-1} \circ p \circ T}$$
on $\Omega_+(q,\Gamma)$ has a removable singularity at $\infty$. It is also non-vanishing and maps $L(q,\Gamma)$ into the unit circle. Hence $\psi$ must be a unimodular constant, by the maximum modulus principle applied to both $\psi$ and $1/\psi$. Since the leading coefficients of $q$ and $p\circ T$ have the same argument, it follows that $\psi$ must be positive and therefore equal to one. This implies that $q=p \circ T$, which proves uniqueness.

This completes the proof of the theorem.

\end{proof}

\acknowledgments{The authors thank the anonymous referees for helpful suggestions.}

\bibliographystyle{amsplain}

\begin{thebibliography}{99}


\bibitem{EKS}
P. Ebenfelt, D. Khavinson and H.S. Shapiro,
Two-dimensional shapes and lemniscates,
\textsl{Complex analysis and dynamical systems {IV}. {P}art 1},
\textbf{553} (2011),
45--59.

\bibitem{FBY}
M. Fortier Bourque and M. Younsi,
Rational Ahlfors functions,
\textsl{Constr. Approx.},
\textbf{41} (2015),
157--183.


\bibitem{HAM}
D.H. Hamilton,
Conformal welding,
\textsl{Handbook of complex analysis: geometric function theory},
\textbf{1} (2002),
137--146.

\bibitem{JEO}
M. Jeong and M. Taniguchi,
Bell representations of finitely connected planar domains,
\textsl{Proc. Amer. Math. Soc.},
\textbf{131} (2003),
2325--2328.

\bibitem{LEV}
O. Lehto and K.I. Virtanen,
\textsl{Quasiconformal mappings in the plane},
Springer-Verlag, New York-Heidelberg,
1973.


\bibitem{MAR}
D.E. Marshall,
Conformal welding for finitely connected regions,
\textsl{Comput. Methods Funct. Theory},
\textbf{11} (2011),
655--669.

\bibitem{PFL}
A. Pfluger,
Ueber die Konstruktion Riemannscher Fl\"achen durch Verheftung,
\textsl{J. Indian Math. Soc.},
\textbf{24} (1961),
401--412.

\bibitem{RIC}
T. Richards,
Conjecture: Every analytic function on the closed disk is conformally a polynomial,
\textsl{Math.stackexchange question},
http://math.stackexchange.com/questions/437598,
(2013).

\bibitem{RI}
T. Richards,
Level curve configurations and conformal equivalence of meromorphic functions,
\textsl{Comput. Methods Funct. Theory},
\textbf{15} (2015),
323--371.


\bibitem{SHM}
E. Sharon and D. Mumford,
2d-shape analysis using conformal mapping,
\textsl{2004 IEEE Computer Society Conference on Computer Vision and Pattern Recognition},
\textbf{2} (2004),
350--357.

\bibitem{YOU}
M. Younsi,
Shapes, fingerprints and rational lemniscates,
\textsl{Proc. Amer. Math. Soc.},
\textbf{144} (2016),
1087--1093.

\end{thebibliography}

\end{document}